\documentclass[12pt]{amsart}

\usepackage{a4wide}
\usepackage{amssymb}
\usepackage{amsfonts}
\usepackage{amsmath}
\usepackage{amsthm}
\usepackage{graphicx}
\usepackage[all]{xy}

\newtheorem{thm}{Theorem}[section]

\newtheorem{lem}[thm]{Lemma}
\newtheorem{prop}[thm]{Proposition}
\theoremstyle{definition}
\newtheorem{defn}[thm]{Definition}
\theoremstyle{remark}
\newtheorem{rem}[thm]{Remark}
\newtheorem{expl}[thm]{Example}
\numberwithin{equation}{section}
\newcommand{\A}{\mathcal{A}}

\newcommand{\qmboxq}[1]{\quad\mbox{#1}\quad}

\def\per{\mathbf{Per}}
\def\fin{\mathbf{Fin}}

\def\A{\mathbb A}
\def\N{\mathbb N}
\def\Z{\mathbb Z}
\def\Q{\mathbb Q}
\def\R{\mathbb R}
\def\C{\mathbb C}
\def\P{\mathbb P}

\def\decpoint{\bullet}

\newcommand{\udots}{\mathinner{%
 \mkern1mu\raise1pt\hbox{.}%
 \mkern2mu\raise4pt\hbox{.}%
 \mkern2mu\raise7pt\vbox{\kern7pt\hbox{.}}\mkern1mu}}

\usepackage{graphicx}% ----------------------------------------------------------------

\begin{document}

\title[beta-expansions for $p$-adic numbers]{beta-expansions of $p$-adic numbers}%

\author[K.~Scheicher]{Klaus~Scheicher{$^1$}}
\address{\tiny $^{1}$Institut f\"ur Mathematik, Universit\"at f\"ur Bodenkultur,
Gregor-Mendel-Stra\ss{}e 33, A-1180 Vienna, Austria}
\email{klaus.scheicher@boku.ac.at}

\author[V.~F.~Sirvent]{V\'{\i}ctor~F.~Sirvent{$^2$}}
\address{\tiny $^{2}$Departamento de Matem\'aticas, Universidad Sim\'on Bol\'{\i}var, Apartado 89000,
Caracas 1086-A, Venezuela}
\email{vsirvent@usb.ve}
%\urladdr{http://www.ma.usb.ve/\textasciitilde  vsirvent}

\author[P.~Surer]{Paul~Surer{$^3$}}
\address{\tiny $^3$Engerthstra\ss{}e 52, A-1200 Vienna, Austria}
\email{palovsky@yahoo.com}

\thanks{The first and second authors were partially supported by the FWF project Nr. P23990.}

\subjclass{11A63,11K41}%
\keywords{$p$-adic numbers, Pisot-Chabauty numbers, beta-expansions, shift radix systems}%

\date{\today}%

% ----------------------------------------------------------------
\begin{abstract}
In the present article, we introduce beta-expansions in the ring $\Z_p$ of $p$-adic integers. We characterise
the sets of numbers with eventually periodic and finite expansions.
\end{abstract}
\maketitle
%\tableofcontents
% ----------------------------------------------------------------
\section{Introduction}
For a real number $\beta>1$, the \emph{beta-transformation} $T=T_\beta$ is defined for
$x\in[0,1]$ by
\begin{equation}
\label{betatrans}
T(x)=\beta x-\lfloor\beta x\rfloor.
\end{equation}
Denote $T^0(x):=x$ and $T^n(x):=T(T^{n-1}(x))$. By iterating this map, we obtain an expansion
$$
x=\frac{x_1}{\beta}+\frac{x_2}{\beta^{2}}+\cdots,
$$
where $x_n=\lfloor\beta T^{n-1}(x)\rfloor$.
We will call the sequence
\begin{equation}
\label{dbeta}
d_\beta(x):=\decpoint x_1 x_2 \cdots
\end{equation}
the \emph{beta-expansion} of $x$.
This setting was introduced by R{\'e}nyi and Parry (\emph{cf.}~\cite{Par60,Ren57}).
For beta-expansions of real numbers, there exist several results for the case when the base
$\beta$ is a Pisot number. Bertrand and Schmidt \cite{Bertrand:77,Schmidt:77} proved
that if $\beta$ is Pisot, then
the set of eventually periodic beta-expansions consists of the non-negative elements of $\Q(\beta)$.
Schmidt proved a partial converse of this statement, namely,
if the non-negative elements of $\Q$ have eventually periodic expansions, then $\beta$ is a Pisot or Salem number.

In the case of Pisot numbers, it is proved that the associated shift space is sofic
(\emph{cf.}~\cite{Bertrand:77,Schmidt:77}).
Soficness could also be proved  for Salem numbers of degree four (\emph{cf.}~\cite{Boyd:87}).

For real beta-expansions
there is a long standing conjecture stated by Schmidt~\cite{Schmidt:77},
that the set of eventually periodic beta-expansions corresponds to the rational
numbers in the interval $[0,1)$ if and only if $\beta$ is a Pisot or Salem number.
It has been proved in \cite{Boyd:87}
that the beta-expansion of $1$ is eventually periodic for Salem numbers of degree four.
It is conjectured by Boyd~\cite{BoydSalemSix},
that this is also true for Salem numbers of degree six, but not for higher degrees.

Following Frougny and Solomyak~\cite{FS92}, we say that an algebraic number $\beta>1$ satisfies the finiteness property (F)
if for each $x \in \mathbb{Z}[\beta^{-1}] \cap [0,1)$ there exists a positive integer $n$ such that $T^n(x)=0$.
Property (F) can only hold for Pisot numbers whose associated shift space is of finite type.
The quadratic Pisot numbers that satisfy (F) are completely described in \cite{FS92}, but from the cubic case on a full
characterisation seems to be very hard. Partial results, conditions and algorithms can be found, among other references, in
\cite{Akiyama:98,srs1,srs2,Akiyama-Rao-Steiner:04,FS92,Surer:07}.
Moreover, beta-expansions have been studied in the context of formal Laurent series over finite fields.
In this context the conjectures of Schmidt and Boyd have been proved (\emph{cf.}~\cite{hbaibmkaouar:06,Scheicher:07}).
%However, both of these conjectures still stand in the context of $p$-adic numbers.

In the present article, we deal with beta-expansions in the ring of the $p$-adic integers.
We characterise the set of numbers with eventually periodic and finite expansions.
In particular in Theorem~\ref{bertrandschmidt}, we prove that for $\beta$ a Pisot-Chabauty number,
the set of eventually periodic beta-expansions is $\Q(\beta)\cap\Z_p$.
This is the equivalent of the result of Bertrand and Schmidt in the context of $p$-adic numbers.
Furthermore, in Theorem~\ref{conversebertrandschmidt}, we prove an equivalent result to Schmidt's partial converse.
In Theorem~\ref{thmfin}, we characterise the set of finite beta-expansions for a family of Pisot-Chabouty numbers.
The theory of beta-expansions of $p$-adic numbers involves techniques from the theory of
beta-expansions of real numbers as well as formal Laurent series.

The article is organised as follows:
in Section 2, we introduce some basic definitions and preliminary results.
In Section 3, we define the beta-expansions for $p$-adic numbers.
In Section 4, we describe the periodic beta-expansions. The main result of this section
is Theorem~\ref{bertrandschmidt}.
In Section 5, we characterise the set finite expansions in Theorem~\ref{thmfin}.
\section{Basic definitions and results}
Let $p$ be a prime and
$$
\A_p:=\{m p^n:m,n\in\Z\}=\Z[\tfrac{1}{p}].
$$
Then $\A_p\subset\Q$ is a principal ring.
The unit group of $\A_p$ is
$
\{p^k:k\in\Z\}
$
and the field of fractions is $\Q$.
Define
$$
\nu_p:\A_p\to\Z\,\cup\{\infty\}
$$
by
$$
\nu_p(x)=
\begin{cases}
\mathrm{max}\{n\in\mathbb{Z}:p^n | x\},& \mbox{if}\quad x \neq 0;\\
\infty, & \mbox{if}\quad x=0.
\end{cases}
$$
Then $\nu_p$ verifies the properties
\begin{equation}
\label{expval}
\nu_p(0)  =\infty,\qquad
\nu_p(xy) =\nu_p(x)+\nu_p(y)
\end{equation}
and
\begin{equation}
\label{non-archimedean}
\begin{aligned}
\nu_p(x+y)&\geq\min\{\nu_p(x),\nu_p(y)\}\qmboxq{with}\\
\nu_p(x+y)&=\min\{\nu_p(x),\nu_p(y)\},\qmboxq{if}\nu_p(x)\ne\nu_p(y).
\end{aligned}
\end{equation}
Therefore $\nu_p(\cdot)$ is an exponential valuation on $\A_p$ (\emph{cf.}~\cite[Chapter~II]{Neukirch:92}).
The $p$-adic norm $|\cdot|_p$ on $\A_p$ is defined by
\begin{equation}
\label{pabs}
|x|_p:=
\begin{cases}
	p^{-\nu_p(x)}, & \text{if } x \neq 0;\\
	0,             & \text{if } x=0.
\end{cases}
\end{equation}
From \eqref{expval} follows
\begin{equation}
\label{nonarchimedean}
|xy|_p =|x|_p|y|_p
\end{equation}
and
\begin{equation}
\label{ultra}
\begin{aligned}
|x+y|_p&\leq\max\{|x|_p,|y|_p)\}\qmboxq{with}\\
|x+y|_p&=\max\{|x|_p,|y|_p)\},\qmboxq{if}|x|_p\ne|y|_p.
\end{aligned}
\end{equation}
Thus, $|\cdot|_p$ is a non-archimedean absolute value on $\A_p$.
Let $|\cdot|_\infty$ be the archimedean absolute value.
Then $|x|_p$ and $|x|_\infty$ satisfy the following product formula
%(\emph{cf.}~\cite[Chapter II,~Theorem~2.1]{Neukirch:92})
$$
\prod_{p\in\P\,\cup\{\infty\}}|x|_p=1\qmboxq{for all}x\in\Q\setminus\{0\}
$$
where $\P$ denotes the set of primes.
The completion of
$\A_p$ with respect to $|\cdot|_p$ is the field $\Q_p$ of
$p$-adic numbers. Thus
$$
\Z\subset\A_p\subset\Q\subset\Q_p.
$$
Each element $x\in\Q_p$ admits a unique expansion of the form
\begin{equation}\label{padicseries}
x=\sum_{n=n_0}^\infty x_{n}p^{n},\qmboxq{such that}n_0\in\Z,\quad x_{n_0}\ne0\qmboxq{and}x_{n}\in\{0,\ldots,p-1\}.
\end{equation}
For expansions of the form \eqref{padicseries}, we will use the notation
$$
x=\cdots p_2 p_1 p_0 \decpoint p_{-1}\cdots p_{n_0}.
$$
The $x_i$ are called the $p$-adic coefficients of $x$.
If we define the extension $\nu_p:\Q_p\to\Z\cup\{\infty\}$ by
$$
\nu_p(x):=
\begin{cases}
n_0, & \text{if } x \neq 0;\\
\infty, & \text{if } x=0,
\end{cases}
$$
then \eqref{pabs} holds also in $\Q_p$.
The ring
$$
\Z_p:=\{x\in\Q_p:|x|_p\leq1\}
$$
is called the ring of $p$-adic integers.
It easily follows that
$$
\Z=\A_p\cap\Z_p=\{x\in \A_p:|x|_p\leq1\}.
$$ Furthermore $\Z_p$  is compact.
The elements of $\Z_p$ can be expressed uniquely in the form \eqref{padicseries} with $n_0\geq0$.
%\begin{prop}
%\label{ndense}
%$\N$ is dense in $\Z_p$.
%\end{prop}
%\begin{proof}
%Let $x=\cdots x_1 x_0\decpoint\in\Z_p$. For $n\in\N$, set $a_n:=\cdots00x_n\cdots x_0\decpoint$.
%Then $a_n\in\N$ and $|x-a_n|_p<p^{-n}$.
%\end{proof}
\begin{defn}
Each $x\in\Q_p$ of the form \eqref{padicseries} has a unique \emph{Artin decomposition}
$$
x=\lfloor x\rfloor_p+\{x\}_p
$$
with
$$
\lfloor x\rfloor_p:=\sum_{n\geq0} x_{n}p^{n}\qmboxq{and}
\{x\}_p:=\sum_{n<0}x_{n}p^{n}.
$$
The number
$\lfloor x\rfloor_p\in\Z_p$ is called \emph{$p$-adic integer part} and
$\{x\}_p\in \A_p\cap[0,1)$ is called \emph{$p$-adic fractional part} of $x$.
\end{defn}
\begin{rem}
\label{abspabsinf}
For $x\in\R$,
let $\lfloor x \rfloor:=\max \{k\in\Z : k\leq x\}$ be the real floor function.
If $x\in \A_p$, it follows that
$\lfloor x\rfloor_p=\lfloor x\rfloor.$ However, for $x\in\Q\setminus \A_p$, this identity generally is not true.
In any case, for $x\in\R\setminus\Q$ or $x\in\Q_p\setminus\Q$, one of these functions is not defined and therefore,
this identity does not hold.
\end{rem}
\begin{defn}
An element $\alpha$ is called \emph{algebraic over $\A_p$}, if
there is a polynomial
\begin{equation}
\label{pxy}
f(x)=a_0+ a_1x+\cdots+ a_n x^n\in \A_p[x]\qmboxq{with} f (\alpha)=0.
\end{equation}
If $f$ is irreducible over $\A_p$, then $f$ is called a minimal polynomial of $\alpha$.
If $a_n=p^k$ for some $k\in\Z$, then $\alpha$ is called an \emph{algebraic integer}.
Since $p^k$ is a unit of $\A_p$, we can assume without loss of generality, that $a_n=1$.
\end{defn}
It turns out that algebraic elements over $\A_p$ are not necessarily contained in $\Q_p$.
In our context, we will only need that $|\cdot|_p$ and $v_p(\cdot)$ can be extended uniquely from $\Q_p$ to all of
its algebraic extensions. This follows from the next
\begin{thm}\label{ext}
\cite[Chapter II,~Theorem~4.8]{Neukirch:92}. Let $K$ be a field which is complete
with respect to $|\cdot|$ and $L/K$ be an algebraic extension of
degree $m$. Then $|\cdot|$ has a unique extension to $L$ defined by
$$
|\alpha|=\sqrt[m]{|N_{L/K}(\alpha)|},
$$
and $L$ is complete with respect to this extension.
\end{thm}
We apply Theorem \ref{ext} to algebraic extensions of
$\Q_p$. Since $\Q_p$ is complete, $|\cdot|_p$ and $v_p(\cdot)$ can be extended uniquely
to each algebraic extension field $L$ of $K=\Q_p$. Thus, every
algebraic element over $\A_p$ can be valuated.

\smallskip
Let $\alpha$ be algebraic over $\A_p$ and
\eqref{pxy} be its minimal polynomial.
The \emph{Newton polygon of $f$} (\emph{cf.}~\cite[Chapter~II]{Neukirch:92}) is defined as the lower convex
hull of the set of points
$$
\{(0,v_p(a_0)),\ldots,(m,v_p(a_m))\}.
$$
The polygon is a sequence of line segments $E_1,E_2,\ldots,E_r$ with
monotically increasing slopes.
\begin{prop}\label{neukirchprop63}
Let
$$
f(x)=a_0+\cdots+a_n x^n,\quad a_0 a_n\ne0,
$$
be a polynomial over the field $K$,
and $K$ be complete with respect to the exponential valuation $v$. Let $w$
be the unique extension of $v$ to the splitting field $L$ of $f$.
\begin{enumerate}
\renewcommand{\labelenumi}{(\roman{enumi})}
\item
If $$\overline{(r,v(a_r))(s,v(a_s))}$$ is a line segment of slope $-m$ occurring in the Newton polygon of $f$, then
$f(x)$
has exactly $s-r$ roots $\alpha_1,\ldots,\alpha_{s-r}$ of value
$$
w(\alpha_1)=\cdots=w(\alpha_{s-r})=m.
$$
\item
Let $E_1,\ldots,E_t$ be the line segments of the Newton polygon. Then, for each $E_j$, there exists a
unique polynomial $g_j(x)\in K[x]$, such that
$$
f(x)=a_n\prod_{j=1}^t g_j(x).
$$
Thus,
$$
g_j(x)=\prod_{w(\alpha_i)=m_j}(x-\alpha_i).
$$
\end{enumerate}
\end{prop}
\begin{rem}
\label{enumeration}
In this article, we will use the convention that, for algebraic elements $\alpha$ over $\A_p$, we will denote by
$\alpha_1,\ldots,\alpha_n$ the non-archimedean conjugates  and by $\alpha_{n+1},\ldots,\alpha_{2n}$
the archimedean conjugates of $\alpha$.
\end{rem}
\begin{lem}
\label{finitenesslem}
Let $A\subset \A_p$.
If $A$ is bounded with respect to $|\cdot|_p$ and $|\cdot|_\infty$, i.e.
$$
\max_{a\in A}|a|_p<\infty\qmboxq{and}
\max_{a\in A}|a|_\infty<\infty,
$$
then $A$ is finite.
\end{lem}
\begin{proof}
If $|a|_p\leq K$ for all $a\in A$, then
$$
A\subset\left\{mp^k: m\in\Z,k=\left\lfloor-\tfrac{\log K}{\log p}\right\rfloor\right\}.
$$
%All these sets are discrete and equispaced with distance $p^k$.
Therefore, if $A$ is bounded with respect to $|\cdot|_\infty$, it can contain only
finitely many points.
\end{proof}
\begin{defn}\label{PisotChabauty}
A Pisot-Chabauty number (for short PC number) is a $p$-adic number $\alpha\in\Q_p$,  such that
\begin{enumerate}
\renewcommand{\labelenumi}{(\roman{enumi})}
\item \ $\alpha_1:=\alpha$ is an algebraic integer over $\A_p$.
\item $|\alpha_1|_p\ >1$ for one non-archimedean conjugate of $\alpha$.
\item $|\alpha_i|_p\ \leq1$ for all non-archimedean conjugates $\alpha_i$, $i\in\{2,\ldots,n\}$  of $\alpha$.
\item $|\alpha_i|_\infty<1$ for all archimedean conjugates $\alpha_i$, $i\in\{n+1,\ldots,2n\}$ of $\alpha$.
\end{enumerate}

\smallskip
\noindent
If condition (iv) is replaced by
\begin{itemize}
\item[(iv)'] $|\alpha_i|_\infty=1$ for all archimedean conjugates $\alpha_i$, $i\in\{n+1,\ldots,2n\}$ of $\alpha$,
\end{itemize}
then $\alpha$ is called a Salem-Chabauty number (for short SC number).
\end{defn}
\begin{rem}
\label{remsalem}
If an archimedean root is located on the complex unit circle, then
its minimal polynomial $f$ is self-reciprocal, i.e. $f(x)=x^n f(\tfrac1x)$.
If there does not exist any root outside the unit circle, there also can not exist
any root inside the unit circle.
Therefore, condition (iv)' is formulated with equality for all archimedean conjugates of $\alpha$.
\end{rem}
\begin{prop}\label{corconj}
Let $\alpha$ be a PC number or SC number. Then $\alpha\in\Q_p$.
\end{prop}
\begin{proof}
Let $\alpha$ and the minimal polynomial $f$ of $\alpha$
be of the form (\ref{pxy}).
Since there is no non-archimedean
conjugate $\alpha_j$, $j\in\{2,\ldots,n\}$ with $|\alpha|_p=|\alpha_j|_p$, the
Newton polygon of $f$ must contain an edge
$$\overline{(i,v_p(a_i))(i+1,v_p(a_{i+1}))},$$
with slope $\nu_p(\alpha)=v_p(a_{i+1})-v_p(a_{i})$
for some $i\in\{0,\ldots,n-1\}$.
By Proposition~\ref{neukirchprop63}~(i), the minimal polynomial must
contain the factor $x-\alpha\in\Q_p[x].$ Thus $\alpha\in\Q_p$.
\end{proof}
\begin{defn}
Let
$$
\mathcal{E}_n:=\{(r_1,\ldots,r_{n})\in\R^n:x^n+r_{n}x^{n-1}+\cdots+r_1\ \mbox{has only complex roots $\alpha$ with $|\alpha|<1$}\}.
$$
\end{defn}
%\noindent
Then $\mathcal{E}_n$ is an open set with $(0,\ldots,0)\in\mathcal{E}_n$.
By considering the Newton polygon
of the minimal polynomial, the following necessary conditions can be derived.
\begin{prop}\label{thmpisot}
Let $\alpha$ be an algebraic integer over $\A_p$ and
\eqref{pxy}
be its minimal polynomial with $a_n=1$.
Then
%\begin{itemize}
%\item
$\alpha$ is a PC number if and only if
\begin{equation}
\label{pcsc}
\nu_p(a_{n-1})\leq\min_{0\leq j\leq n-2} \nu_p(a_j)\qmboxq{and}(a_0,\ldots,a_{n-1})\in\mathcal{E}_n.
\end{equation}
%\item
Moreover, $\alpha$ is a SC number if and only if
\begin{equation}
\nu_p(a_{n-1})\leq\min_{0\leq j\leq n-2} \nu_p(a_j)\qmboxq{and}(a_0,\ldots,a_{n-1})\in\partial{\mathcal{E}}_n.
\end{equation}
%\end{itemize}
In both cases, $\alpha$ is an isolated root with $\alpha\in\Q_p$ and $|\alpha|_p=|a_{n-1}|_p$.
\end{prop}
\begin{proof}
The Proposition follows directly from Proposition~\ref{neukirchprop63}~(ii) and Definition~\ref{PisotChabauty}.
\qedhere
\end{proof}
\begin{expl}
\label{xpl}
The PC numbers of degree one admit the form
$\{x\in \A_p:|x|_p>1\}.$
In order to obtain PC numbers of arbitrary degree,
the following construction may be used.
We consider an irreducible polynomial
$$
f(x):=p^k x^n+a_{n-1}x^{n-1}+\cdots+a_{1}x+a_0\in\Z[x]
$$
such that
$$
\nu_p(a_{n-1})\leq\min_{0\leq j\leq n-2} \nu_p(a_j)
$$
and $k$ is large enough such that
$$
\left(\frac{a_{n-1}}{p^k},\ldots,\frac{a_0}{p^k}\right)\in\mathcal{E}_n.
$$
Such a $k$ exists, since $(0,\ldots,0)$ is an inner point of $\mathcal{E}_n$.
Then the archimedean roots of $f$ fulfill $|\alpha_i|_\infty<1$.
Since $p^k$ is a unit of $\A_p$, the roots of $f$ are algebraic integers over $\A_p$.

By considering the Newton polygon of $f$,
we confirm that $f$ has
one non-archimedean root $\alpha_1$ with $|\alpha_1|_p>1$
and $n-1$ non-archimedean roots $\alpha_i$ with $|\alpha_i|_p\leq1$.
Therefore, $\alpha_1$ is a PC number.
\end{expl}
\begin{rem}
In the context of the present article, it turns out to be more convenient to write the minimal polynomial
of PC numbers or SC numbers in the form
\begin{equation}
\label{minpol}
x^n-a_1 x^{n-1}-\cdots-a_n,\quad a_i\in \A_p[x].
\end{equation}
From now on, we will use this notation for the rest of the article.
\end{rem}
In Theorem~\ref{thmrecurrence}, a method to compute
the $p$-adic coefficients of PC or SC numbers is given. For the proof, we will need the
following auxiliary result.
\begin{lem}\label{simplefact}
If $z,w\in\Q_p$ with $|z|_p=|w|_p$, then
$|z^n-w^n|_p\leq|z-w|_p|z|_p^{n-1}$ for all $n\in\Z$.
\end{lem}
\begin{proof}
The statement is trivial for $n=0$. For $n>0$, we have
\begin{align*}
|z^n-w^n|_p&=|z-w|_p|z^{n-1}+z^{n-2}w+\cdots+w^{n-1}|_p\\
&\leq|z-w|_p
\max_{0\leq j\leq n-1}|z^{n-1-j}w^j|_p\\
&=|z-w|_p|z|^{n-1}_p
\end{align*}
and
\belowdisplayskip=-12pt
\begin{align*}
|z^{-n}-w^{-n}|_p&=|{w^n-z^n}|_p|{z^{-n}w^{-n}}|_p\\
&\leq|w-z|_p|w|_p^{n-1}|z^{-n}w^{-n}|_p\\
&=|z-w|_p|z|_p^{-n-1}.
\end{align*}
\end{proof}
\begin{thm}\label{thmrecurrence}
Let $\alpha$ be a PC number or SC number and
\eqref{minpol}
be its minimal polynomial. Then the recurrence
\begin{equation}
\label{recurrence}
\begin{aligned}
\alpha_0&:=a_1,\\
\alpha_{k+1}&:=a_1+\frac{a_2}{\alpha_k}+\cdots+\frac{a_n}{\alpha_k^{n-1}}\quad\mbox{for }k\geq1
\end{aligned}
\end{equation}
converges to
$$
\lim_{k\to\infty}\alpha_k=\alpha.
$$
\end{thm}
\begin{proof}
First we prove by induction that $|\alpha_k|_p=|a_1|_p$ for all $k\geq0$. For $k=0$ this assertion follows from
the definition.
Let $|\alpha_k|_p=|a_1|_p$.
For $j=2,\ldots,n$, it follows from Proposition~\ref{thmpisot} that $|a_1|_p>1$ and $|a_1|_p\geq|a_j|_p$. Using
\eqref{ultra}, we easily obtain
$$
\left|\frac{a_j}{\alpha_k^{j-1}}\right|_p
=\frac{\left|a_j\right|_p}{\left|\alpha_k\right|_p^{j-1}}
\leq\frac{|a_1|_p}{|a_1|_p}<|a_1|_p.
$$
Thus $|\alpha_{k+1}|_p=|a_1|_p$.
From Lemma~\ref{simplefact}, we get
\begin{eqnarray*}
|\alpha_{k+1}-\alpha_k|_p&=&
\left|
a_2
\left(
\frac{1}{\alpha_k}-\frac{1}{\alpha_{k-1}}
\right)
+\cdots+
a_n
\left(
\frac{1}{\alpha_k^{n-1}}-\frac{1}{\alpha_{k-1}^{n-1}}
\right)
\right|_p\\
&\leq&
\max\left(
\frac{|a_2|_p}{|a_1|_p^2},
\ldots,
\frac{|a_n|_p}{|a_1|_p^n}
\right)|\alpha_k-\alpha_{k-1}|_p.
\end{eqnarray*}
Since $|a_1|_p>1$ and $|a_1|_p\geq|a_j|_p$,
the left factor is constant and less than $1$.
By Banach's fixed-point theorem, the sequence converges to a limit $\alpha$
with
\belowdisplayskip=-12pt
$$
\alpha=a_1+\frac{a_2}{\alpha}+\cdots+\frac{a_n}{\alpha^{n-1}}.
$$
\end{proof}
\begin{rem}
Proposition~\ref{corconj} follows also as a direct consequence from Theorem~\ref{thmrecurrence}.
This gives an alternative proof of Proposition~\ref{corconj}.
\end{rem}
\begin{expl}
Let $f(x)=x^2+\tfrac12 x+\tfrac12$. Then $f$ has two non-archimedean roots $\beta_1,\beta_2$ with
$|\beta_1|_2=2$, $|\beta_2|_2=1$ and two archimedean roots $\beta_3,\beta_4$ with $|\beta_3|_\infty=|\beta_4|_\infty=\tfrac{1}{\sqrt2}$. Thus, the dominant non-archimedean root
$\beta:=\beta_1$ is a PC number.
By Theorem~\ref{thmrecurrence}, the recurrence \eqref{recurrence} converges to
$$
\beta=\cdots110100010010011100011000110110011100111111010010\decpoint1.
$$
%This computation was done using Sage~\cite{sage}.
\end{expl}
\section{beta-expansions of $p$-adic numbers}
Let $\beta\in\Q_p$ with $|\beta|_p>1$ and $\alpha\in\Z_p$. A {\em representation in base $\beta$ (or beta-representation)}
of $\alpha$ is an infinite sequence $(d_i)_{i\geq1}$, $d_j\in \A_p$ such that
\begin{equation}\label{betaexpansion}
\alpha=\sum_{i=1}^\infty\frac{d_i}{\beta^i}.
\end{equation}
A particular beta-representation -- called the {\em beta-expansion} -- can be computed
by the following.
This algorithm works as follows. Set $r^{(0)}=\alpha$ and let
$$
d_k=\{\beta r^{(k-1)}\}_p,\quad r^{(k)}=\lfloor\beta r^{(k-1)}\rfloor_p
$$
for $j\geq1$.
%If $|\beta|_p=p^\ell$, with $\ell<0$, then
%$$
%d_k\in\{p^{\ell}m:m\in\Z,\ 0\leq m<p^{-\ell}\}.
%$$
%\baustelle
If
$$
\mathcal{N}=[0,1)\cap\{x:\in\A_p:|x|_p\leq|\beta|_p\},
$$
then $d_k\in\mathcal{N}$ for all $k\geq1$. The set $\mathcal{N}$ is called the digit set of the beta-expansion.
It is finite with $|\beta|_p$ elements.
%The above procedure yields a representation of $\alpha$ of the form (\ref{betaexpansion}).
Furthermore, $r^{(k)}\in\Z_p$ for all $k\geq1$.
The above procedure defines a mapping
$$d_\beta:\Z_p\rightarrow\mathcal{N}^{\,\N}$$ from $\Z_p$ to $\mathcal{N}^{\,\N}$, the set of one-sided infinite sequences over $\mathcal{N}$, by
\[
d_{\beta}(\alpha):=\decpoint d_1 d_2 \cdots.
\]
It follows that
\[
r^{(k)}=\beta^k \left(\alpha-\sum_{i=1}^k d_i \beta^{-i}\right) .
\]
An equivalent definition is obtained by using the beta-transformation $T:\Z_p\rightarrow\Z_p$,
which is given by the mapping
$z\mapsto\lfloor\beta z\rfloor_p$.
For $k\geq0$, define
$$
T^0(x):=x\qmboxq{and}T^k(x):=T(T^{k-1}(x)).
$$
Then $d_\beta(\alpha)=(d_k)_{k=1}^\infty$ if and only if $d_k=\{\beta T^{k-1}(\alpha)\}_p$
for all $k\geq1$.  If $s$ is the one sided shift defined on $\mathcal{N}^{\,\N}$, i.e.
$$
s(\decpoint d_1 d_2\cdots):=\decpoint d_2 d_3\cdots,
$$
we obtain the following commutative diagram:
\[
\xymatrix{
\Z_p\ar[r]^{T}\ar@<3pt>[d]_{d_\beta} & \Z_p\ar@<3pt>[d]^{d_\beta} \\
\mathcal{N}^{\,\N}\ar[r]_{s}         & \mathcal{N}^{\,\N}.
}
\]
By iteration, we obtain
$
d_\beta(T^k(\alpha))=s^k(d_\beta(\alpha))
$
for all $k\geq0$.
Now let $\alpha\in\Q_p$ with $|\alpha|_p>1$.
Then there is an integer $n>0$ such that $$|\beta|^{n-1}_p<|\alpha|_p\leq|\beta|^{n}_p.$$
If
\begin{align*}
d_\beta(\beta^{-n}\alpha)&=\decpoint d_{-n}\cdots d_{-1} d_0 d_1\cdots,
\intertext{we define}
d_\beta(\alpha)&=d_{-n}\cdots d_{-1} d_0 \decpoint d_1 d_2 \cdots.
\end{align*}
\begin{defn}
Let
\begin{align*}
\per(\beta)&:=\{\alpha\in\Z_p : d_\beta(\alpha)\mbox{ is eventually periodic}\}
\quad\mbox{and}\\
\fin(\beta)&:=\{\alpha\in\Z_p : d_\beta(\alpha)\mbox{ is finite}\}.
\end{align*}
\end{defn}
\noindent
Then, clearly
$$
\fin(\beta)\subset\per(\beta).
$$
\section{Periodic beta-expansions}
The following theorem provides a $p$-adic analogue of the famous Theorem of Bertrand and Schmidt
(\emph{cf.}~\cite{Bertrand:77,Schmidt:77}).
\begin{thm}\label{bertrandschmidt}
Let $\beta$ be a PC number. Then
$
\per(\beta)=\Q(\beta)\cap\Z_p.
$
\end{thm}
\begin{proof}
The proof for $\per(\beta)\subset\Q(\beta)\cap\Z_p$ is trivial.
Therefore, we only prove the opposite inclusion $\Q(\beta)\cap\Z_p\subset\per(\beta)$.
%We enumerate the conjugates as in Remark~\ref{enumeration}.

\smallskip
Let $z\in\Q(\beta)\cap\Z_p$.
First we prove that the orbit of $z$ under $T$ is bounded with respect to $|\cdot|_p$ and $|\cdot|_\infty$.
Let $r_1^{(0)}:=z$, $r_1^{(k)}:=T^k(z)$ and $r_j^{(k)}$ be the corresponding conjugates (with the numeration
defined as in Remark~\ref{enumeration}).
If $(d_k)_{k\geq1}$ is the beta-expansion of $r_1^{(0)}$, then
\begin{equation}\label{konjrest}
r^{(k)}_j=\beta_j^k
\left(
r^{(0)}_j-\sum_{i=1}^k d_i \beta_j^{-i}
\right)
\qmboxq{for all}j\in\{1,\ldots,2n\}.
\end{equation}
Since $r^{(k)}_1\in\Z_p$ for all $k\geq0$, it follows trivially that $|r^{(k)}_1|_p\leq1$.

\smallskip
We consider the non-archimedean conjugates $r^{(k)}_j$, $j\in\{2,\ldots,n\}$.
Since
$$
|\beta_j|_p\leq1\qmboxq{and}|d_\ell|_p\leq|\beta|_p,
$$
it follows by (\ref{non-archimedean}),
\begin{equation}\label{strict}
\begin{aligned}
|r^{(k)}_j|_p
&\leq
\max\left(
|\beta_j^k r_j^{(0)}|_p,\max_{1\leq i\leq k}(|d_i\beta_j^{k-i}|_p)\right) \\
&\leq\max\left(|r_j^{(0)}|_p,|\beta|_p\right)<\infty.
\end{aligned}
\end{equation}
Therefore, $|r^{(k)}_j|_p$ is bounded for all $k$ and $j$.

\smallskip
Now we consider the archimedean conjugates $r_j^{(k)}$, $j\in\{n+1,\ldots,2n\}$.
Let
$$
\gamma:=\max_{n+1\leq j\leq2n}|\beta_j|_\infty.
$$
Since
$$
\gamma<1\qmboxq{and}|d_i|_\infty<1,
$$
it follows from \eqref{konjrest} that
$$
|r_j^{(k)}|_\infty\leq
\gamma^k|r_j^{(0)}|_\infty+
\sum_{i=1}^k \gamma^{k-i}
<\infty
.
$$

\medskip
\noindent
We need a technical result.
\begin{lem}
\label{techlem}
Define the matrices
$$
B_p:=(\beta_j^{-i})\in\Q_p^{\,n\times n}\qmboxq{and}B_\infty:=(\beta_{n+j}^{-i})\in\C^{n\times n}
$$
with $i\in\{0,\ldots,n-1\}$ and $j\in\{1,\ldots,n\}$.
If
$$
R_p^{(k)}:=(r_1^{(k)},\ldots,r_n^{(k)}),\quad
R_\infty^{(k)}:=(r_{n+1}^{(k)},\ldots,r_{2n}^{(k)}),
$$
then
for every $k\geq0$, there exists a unique $n$-tuple
$$
W^{(k)}:=(w_0^{(k)},\ldots,w_{n-1}^{(k)})\in \A_p^n
$$
such that
$$
R_p^{(k)}=q^{-1}W^{(k)}B_p\quad\mbox{and}\quad
R_\infty^{(k)}=q^{-1}W^{(k)} B_\infty
$$
with $q\in\N$.
\end{lem}
\begin{proof}
The Lemma is proved is by induction.
For $k=0$, it follows from
$$
\beta^n_j=a_1\beta^{n-1}_j+\cdots+a_n,
$$
that
\begin{equation}
r_j^{(0)}
=q^{-1}\sum_{i=0}^{n-1} z_i \beta^{i}_j
=q^{-1}\sum_{i=1}^n w^{(0)}_i \beta^{-i}_j.
\end{equation}
If $k>0$, then
\begin{equation}
\begin{aligned}
r_j^{(k)}&=\beta_j r_j^{(k-1)}-d_k\\
&=q^{-1}\left(\beta_j\sum_{i=1}^n w^{(k-1)}_i \beta_j^{-i}-q d_k\right)\\
&=q^{-1}\left(w^{(k-1)}_1-q d_k+\sum_{i=1}^{n-1} w^{(k-1)}_{i+1} \beta_j^{-i}\right)\\
&=q^{-1}\left(\sum_{i=1}^n w^{(k)}_i \beta_j^{-i}\right).
\end{aligned}
\end{equation}
The Lemma follows from the definition of the matrices $B_p$ and $B_\infty$.
\end{proof}
We continue now with the proof of Theorem~\ref{bertrandschmidt}.
On $\A_p^n$, we define the vector norms
$$
\|(a_1,\ldots,a_n)\|_p:=\max_{1\leq i\leq n}|a_i|_p\qmboxq{and}
\|(a_1,\ldots,a_n)\|_\infty:=\max_{1\leq i\leq n}|a_i|_\infty.
$$
The induced matrix norms on $\A_p^{n\times n}$ are given by
$$
\|A\|_p = \max \limits _{1 \leq i,j \leq n}  | a_{ij} |_p\qmboxq{and}
\|A\|_\infty = \max \limits _{1 \leq i \leq n} \sum _{j=1} ^n | a_{ij} |_\infty.
$$
Since $B_p$ and $B_\infty$ are
invertible, and the $r_j^{(k)}$ are bounded %w.r.t. $|\cdot|_p$ and $|\cdot|_\infty$
for all $j\in\{1,\ldots,2n\}$, it follows by Lemma~\ref{techlem} that
\begin{align*}
\max_{0\leq j\leq n-1}|w_{j}^{(k)}|_p
&=\left\|W^{(k)}\right\|_p=\left\|q R_p^{(k)} B_p^{-1}\right\|_p\\
&\leq\left\|R_p^{(k)}\right\|_p\left\|q B_p^{-1}\right\|_p<\infty
\intertext{and analogously,}
\max_{0\leq j\leq n-1}|w_{j}|_\infty&<\infty.
\end{align*}
Since $w_j\in \A_p$ for all $j$,
by Lemma~\ref{finitenesslem}, the $w_j$ must be contained in a finite subset of $\A_p$.
Therefore,
the $r_1^{(k)}$ must be contained in a finite subset of $\Q(\beta)$
und thus, there exists an $m$ such that $r_1^{(k+m)}=r_1^{(k)}$ for all $k$ large enough.
This implies that the beta-expansion of $z$ is eventually periodic.
\end{proof}
\begin{expl}
Since $\beta=\tfrac1p$ is the root of
$f(x)=px-1$, one easily verifies that $\tfrac1p$ is a PC number.
Since $\Q(\tfrac1p)=\Q$, Theorem~\ref{bertrandschmidt} implies that
the numbers of $\Q\cap\Z_p$ admit eventually periodic $p$-adic expansions.
\end{expl}
\begin{thm}
\label{conversebertrandschmidt}
Let $S\subset\N$ be an infinite set of positive integers, such that $1\in S$.
If $S\subset\per(\beta)$,
then $\beta$ is a PC number or SC number.
\end{thm}
\begin{proof}
Since $1\in S\subset\per(\beta)$, it follows that
$
d_\beta(1)=\decpoint d_1 d_2 \cdots
$
is eventually periodic.
If $k$ is the length of the preperiod and $\ell$ is the length of the period, then
$d_{i+\ell}=d_i$ for all $i\geq k+1$.
Therefore,
$$
d_\beta(1)=\decpoint
d_1\cdots d_k d_{k+1}\cdots d_{k+\ell}d_{k+1}\cdots d_{k+\ell}\cdots
$$
and thus,
\begin{equation}
\label{pol}
\left(
\beta^\ell-1
\right)\beta^{k}
\left(1-\frac{d_1}{\beta}-\cdots-\frac{d_k}{\beta^{k}}\right)-
\beta^{\ell}
\left(\frac{d_{k+1}}{\beta}+\cdots+\frac{d_{k+\ell}}{\beta^{\ell}}\right)=0.
\end{equation}
Note that in the case that $d_\beta(1)$ is finite, the second summand of \eqref{pol} is zero.
Since $d_i\in \A_p$, it follows that $\beta$ is an algebraic integer over $\A_p$.

\smallskip
We consider the Newton polygon of \eqref{pol}. Since $d_1=\{\beta\cdot1\}_p$, it follows that
$$
\nu_p(1)=0,\quad
\nu_p(-d_1)=\nu_p(\beta)<0
\qmboxq{and}
\nu_p(d_j)\geq\nu_p(\beta)\qmboxq{for}
j\geq2.
$$
Therefore, the Newton polygon contains one edge with slope $-\nu_p(\beta)>0$ and all other edges have slopes $\leq0$.
By Proposition~\ref{neukirchprop63}~(i), there exists one non-archimedean root $\alpha$ of \eqref{pol} with $\nu_p(\alpha)=\nu_p(\beta)<0$ and all other non-archimedean roots $\tilde\alpha$  have $\nu_p(\tilde\alpha)\geq0$.
Since \eqref{pol} is a multiple of the minimal polynomial of $\beta$, it follows that $\alpha=\beta$ and $|\tilde\alpha|\leq1$.

\smallskip
Now we examine the archimedean conjugates of $\beta$. Assume that there exists an ar\-chi\-me\-de\-an conjugate $\beta_j$,
$j\in\{n+1,\ldots,2n\}$, such that $|\beta_j|_\infty>1$.
Since $S$ is infinite, there exists a $k\in S$, with
\begin{equation}
\label{kbound}
k>\frac{1}{|\beta_j|_\infty-1}.
\end{equation}
It follows that
$$
k=
\sum_{i=1}^s\frac{e_{i}}{\beta^{i}}
+\frac{r_1^{(s)}}{\beta^{s}}
=
\sum_{i=1}^s\frac{e_i}{\beta_j^i}
+\frac{r_j^{(s)}}{\beta_j^s}
\qmboxq{for all}s\geq0.
$$
Since $k\in\per(\beta)$, the sequence
$d_\beta(k)=\bullet e_1 e_2\cdots$
is eventually periodic and thus, the
$r_1^{(s)}$ and $r_j^{(s)}$ can take only finitely many values.
Therefore, the sequence
$$
k=
\sum_{i=1}^\infty\frac{e_i}{\beta_j^i}
$$
converges. Since $|e_i|_\infty<1$, it follows that
$$
k\leq
\frac{1}{|\beta_j|_\infty-1},
$$
which contradicts \eqref{kbound}.
\end{proof}
% ----------------------------------------------------------------
\section{Shift radix systems and finite beta-expansions}
Shift radix systems, were introduced in \cite{srs1} in order to
provide a unified notation for
two well known types of number systems, namely,
canonical number systems as well as beta-expansions of real numbers.

\smallskip
Let $n\ge 1$ be an integer and
${\bf r}=(r_1,\dots,r_n) \in \R^n$.
Let $\lfloor\cdot\rfloor$ and $\lceil\cdot\rceil$ denote the real floor and ceiling functions respectively.
To the vector ${\bf r}$,
we associate the mapping
$
\tilde{\tau}_{{\bf r}} :  \Z^n \to \Z^n
$
as follows:
if
${\bf z}=(z_1,\dots,z_n) \in \Z^n$ then let
\begin{equation}
\label{tildetau}
\tilde{\tau}_{{\bf r}}({\bf z}):=(z_2,\ldots,z_{n},-\lfloor {\bf r}
{\bf z}\rfloor),
\end{equation}
where ${\bf r}{\bf z}$ is the
the euclidian inner
product of ${\bf r}$ and ${\bf z}$, i.e.
$$
{\bf r} {\bf z}=r_1 z_1 +\ldots + r_n z_n.
$$
Then $(\Z^n,\tilde{\tau}_{{\bf r}})$ is called {\em a shift radix system} (for short SRS) on $\Z^n$.
Note that $\mathbf{0}=(0,\dots,0)$ is a fixed point of
$\tilde{\tau}_{{\bf r}}$. We say that the orbit of $\mathbf{z}$ under $\tilde{\tau}_{{\bf r}}$ ends up in $\mathbf{0}$,
if there exists a $k\geq0$ such that $\tilde{\tau}_{{\bf r}}^k(\mathbf{z})=\mathbf{0}$.

\smallskip
It is an important problem for canonical number systems as well as beta-expansions
of real numbers to determine wether each number admits a finite expansion.
In SRS language, this translates to the following question:

\bigskip
\noindent
\emph{For which $\mathbf{r}\in\Z^n$, we have that all orbits of
$(\Z^n,\tilde{\tau}_{{\bf r}})$ end up in
$\mathbf{0}$?
}

\bigskip
In the present section, we will prove that
beta-expansions of $p$-adic numbers can be described by a slight variation of this dynamical system,
namely by $(\Z^n,\tau_{{\bf r}})$ with
\begin{equation*}
%\label{tau}
\tau_\mathbf{r}(\mathbf{z})=(z_2,\ldots,z_n,-\lceil \mathbf{r}\mathbf{z}\rceil).
\end{equation*}
In Proposition~\ref{srslemma}, we will prove that for a fixed given vector
$\mathbf{r}\in\R^n$, both of these systems show exactly the same behaviour.
Let
\begin{equation}
\label{dn0}
\begin{aligned}
\mathcal{D}_{n}&:=\{\mathbf{r}\in\R^d:\forall\ \mathbf{z}\in\Z^d,
\mbox{ the sequence }(\tilde\tau_\mathbf{r}^k(\mathbf{z}))_{k\geq1}\mbox{ is eventually periodic}\},\\
{\mathcal D}_{n}^0&:=
\{\mathbf{r}\in\R^d:\forall\ \mathbf{z}\in\Z^d, \exists
k>0:\tilde\tau_\mathbf{r}^k(\mathbf{z})=\mathbf{0}\}.
\end{aligned}
\end{equation}
In \cite{srs1,srs2}, it is proved that the inclusions
$\mathcal{D}_{n}^0\subset\mathcal{D}_{n}$ and
$\mathcal{E}_n\subset\mathcal{D}_n\subset\overline{\mathcal{E}}_n$
hold.
It is easy to verify that $\mathcal{D}_1^0=[0,1)$ and $\mathcal{D}_1=[-1,1]$. However,
a full description of $\mathcal{D}_2^0$ and $\mathcal{D}_2$ is already unknown.
The sets $\mathcal{D}_n^0$ and $\mathcal{D}_n$ have been extensively studied in the last years
(\emph{cf.}~\cite{srs1,srs2,srs3,srs4,ssrs1,bkt12,ssrs2,Scheicher:07,weitzer}).
In Figure~\ref{d20}, an approximation of $\mathcal{D}_2^0$ is shown.
\begin{figure}[tbp]
\includegraphics[width=11cm]{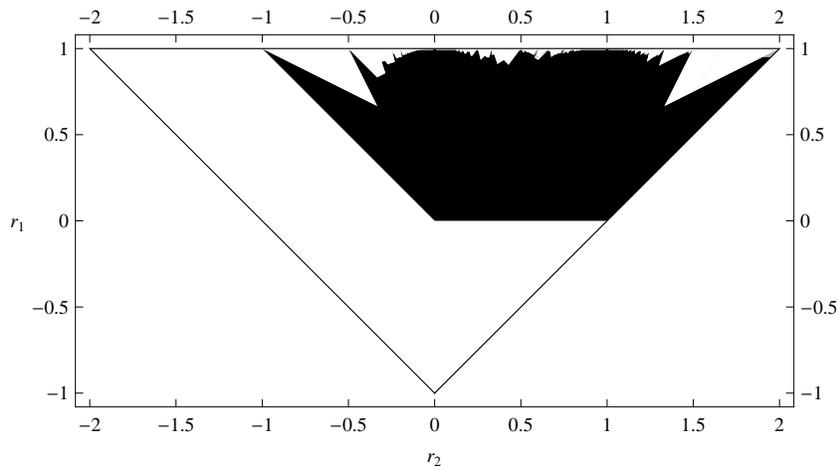}
\caption{
\label{d20}The set $\mathcal{D}_2^0$ (in black) inside $\mathcal{D}_2$ (the interior of the triangle).
}
\end{figure}

It is conjectured, that $\mathcal{D}_n^0\subset\mathcal{E}_n$  which is equivalent to
$\mathcal{D}_n^0\cap\partial\mathcal{E}_n=\emptyset$ (\emph{cf.}~\cite{srs2}).
Up to now, this conjecture has been proved only for $n\leq3$ (\emph{cf.}~\cite{srs2,bkt12}).
%Some special points of the boundary of $\mathcal{D}_2$ are studied in \cite{ABPS2,ABPS1}.
For a very recent survey on SRS, we refer to \cite{KT2013}.
\begin{prop}
\label{srslemma}
Let ${\bf r} \in \mathbb{R}^n$. All orbits of $\left(\mathbb{Z}^n,\tau_{\bf r}\right)$ end up in ${\bf 0}$  if and only if
${\bf r} \in \mathcal{D}_n^0$.
\end{prop}
\begin{proof}
At first we show that $\tau_{\bf r}({\bf z})= - \tilde\tau_{\bf r}(-{\bf z})$ for all ${\bf z} \in \mathbb{Z}^n$. Indeed, consider an arbitrary ${\bf z}=(z_1,\ldots,z_n) \in \mathbb{Z}^n$ and note that for any real number $x$ we have $\left\lceil x \right\rceil = -\left\lfloor -x \right\rfloor$.
We obtain
\begin{equation*}
%\label{tautildetau}
\begin{aligned}
\tau_{\bf r}({\bf z})
&= \left(z_2,\ldots,z_n,-\left\lceil {\bf rz} \right\rceil\right)= \left(z_2,\ldots,z_n,\left\lfloor -{\bf rz} \right\rfloor\right)\\
&=-\left(-z_2,\ldots,-z_n, -\left\lfloor -{\bf rz} \right\rfloor\right)=-\tilde\tau_{\bf r}(-{\bf z}).
\end{aligned}
\end{equation*}
By induction, it follows  that
\begin{equation}
\label{tautildetau1}
\tau^k_{\bf r}({\bf z})=-\tilde\tau^k_{\bf r}(-{\bf z}).
\end{equation}
Now we easily see that all orbits of $\left(\mathbb{Z}^n,\tau_{\bf r}\right)$ end up in ${\bf 0}$ if and only if all orbits of $\left(\mathbb{Z}^n,\tilde\tau_{\bf r}\right)$ do so.
\end{proof}
In order to prove Theorem~\ref{thmfin}, which is the main result of this section, we need the following preliminary results.
\begin{prop}
\label{lemabs}
If $U,V\in\Q_p$, then
$$
|U-\{U+V\}_p|_p=|\lfloor U+V\rfloor_p-V|_p=
\begin{cases}
|\lfloor U\rfloor_p|_p,&\mbox{if}\quad|V|_p\leq1;\\
|\{V\}_p|_p,           &\mbox{if}\quad|V|_p>1;
\end{cases}
$$
or equivalently,
$$
\nu_p(U-\{U+V\}_p)=\nu_p(\lfloor U+V\rfloor_p-V)=
\begin{cases}
\nu_p(\lfloor U\rfloor_p),&\mbox{if}\quad\nu_p(V)\geq0;\\
\nu_p(\{V\}_p),           &\mbox{if}\quad\nu_p(V)<0.
\end{cases}
$$
\end{prop}
\begin{proof}
Straight forward.
\end{proof}
\begin{prop}\label{purely}
Let $\beta$ be an arbitrary element of $\Q_p$ with $|\beta|_p>1$,
and let $z\in \A_p[\beta^{-1}]\cap\Z_p$ have purely
periodic beta-expansion with period $\ell$. Then $z\in \A_p[\beta]\cap\Z_p$.
\end{prop}
\begin{proof} Assume $z\in\ \A_p[\beta^{-1}]\cap\Z_p$
is purely periodic with period $\ell$. Let $d_\beta(z)=\decpoint d_1 d_2\ldots$.
Since $z\in \A_p[\beta^{-1}]$, there is an $m$ such that
$\beta^{m\ell}z\in \A_p[\beta]$. Therefore
\belowdisplayskip=-12pt
$$
z=\beta^{m\ell}z-d_1\beta^{m\ell-1}-\cdots-d_{m\ell}\in \A_p[\beta].
$$
\end{proof}
Now we are able to state the main result of this section.
\begin{thm}\label{thmfin}
Let $\beta\in\Q_p$ with $|\beta|_p>1$.
Then
\begin{equation}
\label{propf}
\tag{F}
\fin(\beta)=\A_p[\beta^{-1}]\cap\Z_p,
\end{equation}
if and only if $\beta$ is an PC number and its minimal polynomial
\begin{equation}
\label{minpolfin}
x^n-a_1 x^{n-1}-\cdots-a_n,\quad a_i\in \A_p[x]
\end{equation}
fulfills
\begin{subequations}
\label{subeqns}
\begin{equation}
\label{aibound}
\max_{2\leq i\leq n}|a_i|_p<|a_1|_p\quad\mbox{and}
\end{equation}
%\begin{equation}
%\label{aibound}
%\min_{2\leq i\leq n}\nu_p(a_i)>\nu_p(a_1)\quad\mbox{and}
%\end{equation}
\begin{equation}
\label{aindn0}
-\mathbf{a}\in\mathcal{D}_n^0,\qmboxq{where}\mathbf{a}:=(a_n,\ldots,a_1).
\end{equation}
\end{subequations}
\end{thm}
\begin{proof}
We will prove first that \eqref{subeqns} implies \eqref{propf}.
Since it is trivial that $\fin(\beta)\subset \A_p[\beta^{-1}]\cap\Z_p$, we will
prove only the opposite inclusion.
The proof runs in two steps. Condition
\eqref{aibound} ensures that
after some preliminary phase, the dynamical system reduces
to a classical SRS. Condition
\eqref{aindn0} ensures that all orbits of this SRS end up in $\bf{0}$.

\smallskip
Let $z\in \A_p[\beta^{-1}]\cap\Z_p$.
From Theorem~\ref{bertrandschmidt}, it
follows that
$$
\A_p[\beta^{-1}]\cap\Z_p\subset\Q(\beta)\subset\per(\beta).
$$
Thus, $z$ has an eventually periodic expansion.
Without loss of generality, we can assume that $z\in \A_p[\beta]$.
Otherwise, we replace $z$ by an element from the orbit with purely periodic expansion
and apply Proposition~\ref{purely}.
Let
$$
\mathcal{B}=\{1,\beta,\ldots,\beta^{n-1}\}\qmboxq{and}\mathcal{V}=\{v_1,\ldots,v_{n}\}
$$
where
\begin{subequations}
\label{base}
\begin{align}
\label{base1}
v_{j}
&=\beta^{n-j}-a_1\beta^{n-j-1}-\cdots-a_{n-j}\\
\label{base2}
&=\frac{a_{n-j+1}}{\beta}+\cdots+\frac{a_n}{\beta^j}
\end{align}
\end{subequations}
for $j\in\{1,\ldots,n\}$.
Note $v_n=1$.
Let
$\mathbf{v}=(v_1,\ldots,v_n)$.
Then both $\mathcal{B}$ and $\mathcal{V}$ are two different bases of $\A_p[\beta]$
considered as a lattice over $\A_p$.
Using (\ref{base}),
the coordinates with respect to $\mathcal{V}$
can be computed from the coordinates with respect to $\mathcal{B}$
by a linear system of equations.
%Hence, for every $z\in \A_p[\beta]$ with
%$$
%z=\bar z_0+\bar z_1\beta+\cdots+\bar z_{n-1}\beta^{n-1},
%$$
%there exist
%$z_1,\ldots,z_{n}\in \A_p$ such that
%$$
%z=z_1 v_1+\cdots+z_n v_n.
%$$
%The $z_i$ are given by the solution of the system of linear equations
%\baustelle
%$$
%\left(
%\begin{array}{ccccc}
%  -a_{n-1} & \cdots & -a_2   & -a_1   & 1 \\
%  -a_{n-2} & \cdots & -a_1   & 1      & 0 \\
%  \vdots   & \udots & \udots & \udots & \vdots \\
%  -a_1     & \udots & \udots &        & \vdots \\
%  1        & 0      & \cdots & \cdots & 0
%\end{array}
%\right)
%\left(
%\begin{array}{c}
%z_{1}\\
%\vdots\\
%\vdots\\
%\vdots\\
%z_{n}
%\end{array}
%\right)=
%\left(
%\begin{array}{c}
%\bar z_0\\
%\vdots\\
%\vdots\\
%\vdots\\
%\bar z_{n-1}
%\end{array}
%\right).
%$$
In this way, we define a bijection $\varphi:\A_p[\beta]\to \A_p^n$ by $$ \varphi(z):=(z_1,\ldots,z_{n}). $$
By construction, $\varphi^{-1}$ is given by
$$
\varphi^{-1}\left((z_1,\ldots,z_{n})\right)=z_1 v_1+\cdots+z_n v_n.
$$
In base $\mathcal{V}$, multiplication by $\beta$ is represented by the matrix
%\[
%M=
%\left(
%\begin{array}{ccccc}
%0      & 1      & 0      & \cdots & 0      \\
%\vdots & \ddots & \ddots & \ddots & \vdots \\
%\vdots &        & \ddots & \ddots & 0      \\
%0      & \cdots & \cdots & 0      & 1      \\
%a_n    & \cdots & \cdots & a_2    & a_1
%\end{array}
%\right).
%\]
\[
M:=
\left(
\begin{array}{ccccc}
0      & \cdots & \cdots & 0      & a_n    \\
1      & \ddots &        & \vdots & \vdots \\
0      & \ddots & \ddots & \vdots & \vdots \\
\vdots & \ddots & \ddots & 0      & a_2    \\
0      & \cdots & 0      & 1      & a_1
\end{array}
\right)
\]
such that
$$
\varphi(\beta z)=\varphi(z)M.
$$
Let
$
\mathbf{e}:=\varphi(1)=(0,\ldots,0,1).
$
Since
$$
T(z)=\beta z-\{\beta z\}_p,
$$
the beta-transformation with respect to $\mathcal{V}$ takes the form
\begin{equation}\label{TV}
\sigma:\A_p^n\rightarrow \A_p^n
\qmboxq{with}
\mathbf{z}\mapsto \mathbf{z}\,M-\{\mathbf{z}\,M\mathbf{v}^\top\}_p\mathbf{e}.
\end{equation}
Then
$$
\varphi(T(z))=\sigma(\varphi(z))\qmboxq{and}T(\varphi^{-1}(\mathbf{z}))=\varphi^{-1}(\sigma(\mathbf{z})),
$$
which is indicated in the following commutative diagram:
%\comment{We need both directions of the diagram for the proof.}
\[
\xymatrix{
\A_p[\beta] \ar[r]^{T}      \ar@<3pt>[d]^{\varphi}      & \A_p[\beta] \ar@<3pt>[d]^{\varphi} \\
\A_p^{n}    \ar[r]_{\sigma} \ar@<3pt>[u]^{\varphi^{-1}} & \A_p^{n}    \ar@<3pt>[u]^{\varphi^{-1}}.
}
\]
Substituting $M$, $\mathbf{v}$ and $\mathbf{e}$ into (\ref{TV}),
we can express the map $\sigma$ as follows:
\begin{equation*}
\begin{aligned}
&\sigma:(z_1,\ldots,z_{n})\mapsto(z_2,\ldots,z_{n+1})\quad\mbox{with}\\
&z_{n+1}=a_n z_1+\cdots+a_1 z_{n}-\{a_n z_1+\cdots+a_1 z_{n}+v_1 z_2 +\cdots+v_{n-1}z_{n}\}_p.
\end{aligned}
\end{equation*}
The $k$th iterate of $\sigma$ is given by
$$
\sigma^k((z_1,\ldots,z_{n}))=(z_{k+1},\ldots,z_{k+n}),
$$
where
\begin{equation*}
%\label{tau}
z_{k+n}=U_k-\{U_k+V_k\}_p
\end{equation*}
with
\begin{equation}
\label{UV}
\begin{aligned}
U_k&:=a_n z_k+\cdots+a_1 z_{k+n-1}\quad\mbox{and}\\
V_k&:=v_1z_{k+1} +\cdots+v_{n-1}z_{k+n-1}.
\end{aligned}
\end{equation}
From \eqref{aibound}, \eqref{base2} and $|a_1|_p=|\beta|_p,$ it follows that
\begin{equation}
\label{vibound}
\max_{1\leq i\leq n-1}|v_i|_p<1
\end{equation}
and therefore,
$$
|V_k|_p\leq\max_{1\leq i\leq n-1}|v_{i}z_{k+i}|_p
<\max_{k+1\leq i\leq k+n-1}|z_{i}|_p.
$$
By Proposition~\ref{lemabs}, we get the following implications:
\begin{itemize}
\renewcommand{\labelenumi}{(\roman{enumi})}
\itemsep3pt
\item If $|V_k|_p>1$, it follows that $$|z_{k+n}|_p=|V_k|_p<\max_{k+1\leq i\leq k+n-1}|z_{i}|_p.$$
\item If $|V_k|_p\leq1$, it follows that $$|z_{k+n}|_p=|\lfloor U_k\rfloor_p|_p\leq1.$$
\end{itemize}
Therefore, there must exist some $k_0\geq1$ such that
\begin{equation}
\label{zibound}
\max_{0\leq j\leq n-1}|z_{k+j}|_p\leq1
\end{equation}
for all $k>k_0$.
Since $z_k\in \A_p$, it follows that
\begin{equation}
\label{zkinint}
\sigma^{k}((z_1,\ldots,z_n))=
(z_{k+1},\ldots,z_{k+n})\in\Z^n
\end{equation}
holds for all $k\geq k_0$.
From
\eqref{zibound},
it follows by \eqref{vibound}  that
$|V_k|_p<1.$
Therefore, $$\{U_k+V_k\}_p=\{U_k\}_p$$ and thus,
$$
z_{k+n}=U_k-\{U_k+V_k\}_p=\lfloor U_k\rfloor_p.
$$
Since $U_k\in \A_p$, it follows that
$$
\lfloor U_k\rfloor_p=\lfloor U_k\rfloor=-\lceil-U_k\rceil
$$
(\emph{cf.}~Remark~\ref{abspabsinf}).
Thus
\begin{align*}
\sigma((z_k,\ldots,z_{k+n-1}))%&=(z_{k+1},\ldots,z_{k+n-1},\lfloor U_k\rfloor)\\
&=(z_{k+1},\ldots,z_{k+n-1},-\lceil-U_k\rceil)\\
&=\tau_\mathbf{-a}((z_k,\ldots,z_{k+n-1})).
\end{align*}
By \eqref{tautildetau1}, it follows that
\begin{align*}
\sigma^k((z_1,\ldots,z_n))
&=\tau_\mathbf{-a}^{k-k_0}(\sigma^{k_0}((z_1,\ldots,z_n)))\\
&=-\tilde\tau_\mathbf{-a}^{k-k_0}(-\sigma^{k_0}((z_1,\ldots,z_n)))
\end{align*}
for all $k\geq k_0$.
Since $\mathbf{-a}\in\mathcal{D}_n^0$, there must exist some $k\geq k_0$ with
$$
\sigma^k((z_1,\ldots,z_n))=\bf{0}.
$$
Thus, each $z\in \A_p[\beta^{-1}]\cap\Z_p$
admits a finite representation with respect to $\beta$.

\bigskip
Now we will prove the converse direction of Theorem~\ref{thmfin}. Suppose that
$$\A_p[\beta^{-1}]\cap\Z_p=\fin(\beta).$$
Since
$
\N\subset\fin(\beta)\subset\per(\beta),
$
it follows by
Theorem~\ref{conversebertrandschmidt}, that $\beta$ is a PC number or SC number.
Without loss of generality, we can assume that the minimal polynomial has the form \eqref{minpolfin}.
In order to exclude the case of SC numbers, we distinguish the following cases.

\smallskip
\begin{enumerate}
\renewcommand{\labelenumi}{(\roman{enumi})}
\itemsep3pt
  \item For $n=1$, there do not exist any SC numbers.
  \item For $n=2$, it is proved in \cite{srs2}, that $\mathcal{D}_2^0\cap\partial\mathcal{E}_2=\emptyset$.
  If there exists an archimedean root of \eqref{minpolfin} that is located on the complex unit circle,
  it follows that
  $-\mathbf{a}\in\partial\mathcal{E}_2$, which contradicts \eqref{aindn0}.
  \item Let $n>2$. If \eqref{minpolfin} is the minimal polynomial of a SC number, it must be
  self-reciprocal (\emph{cf.}~Remark~\ref{remsalem}). Then $a_1=a_{n-1}$, which contradicts \eqref{aibound}.
\end{enumerate}

\smallskip
\noindent
Thus, $\beta$ is a PC number.
We will prove now that each of the converses of \eqref{aibound} and \eqref{aindn0} contradict \eqref{propf}.

\smallskip
In order to prove the necessity of \eqref{aibound},
we follow the proofs of
\cite[Lemma~2.4]{ST2003} and
\cite[Theorem~5.4]{Scheicher:07}.
Suppose that
\begin{equation}
\label{aibound2}
\max_{2\leq i\leq n}|a_i|_p\geq|a_1|_p
\qmboxq{or equivalently}
\min_{2\leq i\leq n}\nu_p(a_i)\leq\nu_p(a_1).
\end{equation}
We will construct an element $z\in \A_p[\beta]$
that does not have a finite representation.

At first, we will prove that there exists some index $h\in\{1,\ldots,n-1\}$, such that
$\nu_p(v_h)\leq0.$
Define
\begin{equation}
\label{h0}
h_0
:=\max%_{2\leq i\leq n}
\left\{
h\in\{2,\ldots,n\}
:\nu_p(a_h) =
\min_{2\leq j\leq n} \nu_p(a_j)
\right\},
\end{equation}
i.e. $h_0$ is the minimal index $h\in\{2,\ldots,n\}$ such that $\nu_p(a_h)$ attains its minimal value.
Then, by \eqref{aibound2} and \eqref{h0}, it follows that $2\leq h_0\leq n$ and thus
\begin{equation}
\label{nh01}
1\leq n-h_0+1\leq n-1.
\end{equation}
By \eqref{h0}, it follows that
$\nu_p(a_j)>\nu_p(a_{h_0})$
for all $j>h_0$ and thus, by \eqref{base},
\begin{align*}
\nu_p(v_{n-h_0+1})&=
\nu_p
\left(
\frac{a_{h_0}}{\beta}+\cdots+\frac{a_n}{\beta^{n-h_0+1}}
\right)\\
&=\nu_p(a_{h_0})-\nu_p(\beta)\leq\nu_p(a_{h_0})-\nu_p(a_1)\leq0.
\end{align*}
If
\begin{equation}
\label{i0}
i_0
:=\min%_{2\leq i\leq n}
\left\{
i\in\{1,\ldots,n-1\}
:\nu_p(v_i) =
\min_{1\leq j\leq n-1} \nu_p(v_j)
\right\},
\end{equation}
then
$\nu_p(v_{i_0})\leq0.$
For $(z_1,\ldots,z_n)\in\A_p^n$, define
\begin{eqnarray*}
\lefteqn{j_0((z_1,\ldots,z_n)):=}\\
&:=&
\begin{cases}
\displaystyle
\max
\left\{ i\in\{2,\ldots,n\}
:\nu_p(z_i)=\min_{i_0+1\leq j\leq n}\nu_p
(z_j)\right\},&\hbox{if }
\displaystyle\min_{i_0+1\leq j\leq n}\nu_p(z_j)<0;\\
0,&\hbox{otherwise}.
\end{cases}
\end{eqnarray*}
Hence,
if $j_0((z_1,\ldots,z_n))>0,$ it follows that $(z_1,\ldots,z_n)\ne\mathbf{0}.$
If
$$
\mathbf{z}^{(0)} :=
(z_1^{(0)},\ldots,z_n^{(0)})=(0,\ldots,0,p^{-1}),
$$
then
$
j_0(\mathbf{z}^{(0)})=n
$
and thus,
\begin{equation}\label{s1}
i_0+1\leq j_0(\mathbf{z}^{(0)})\leq n.
\end{equation}
For $k\geq0$, let
$$
\mathbf{z}^{(k)}
:=(z_1^{(k)},\ldots,z_n^{(k)})
=(z_{k+1},\ldots,z_{k+n})
=\sigma^k(\mathbf{z}^{(0)}).
$$
We will show that $\mathbf{z}^{(0)}$ has an infinite representation by proving that
\begin{equation}\label{s2}
i_0+1\leq{j_0}(\mathbf{z}^{(k)})\leq n%\qmboxq{for all}k\geq0
\end{equation}
for all $k\geq0$ and thus,
$$
\mathbf{z}^{(k)}=\sigma^k(\mathbf{z}^{(0)})\ne\mathbf{0}.
$$
We will prove \eqref{s2} by induction.
By \eqref{s1}, equation \eqref{s2} holds for
$k=0$. Thus, we can proceed to the induction step.
Suppose that
\eqref{s2} holds for $k-1\geq0$ and note that
\begin{equation}\label{s3}
\mathbf{z}^{(k)} =(z_1^{(k)},\ldots,z_n^{(k)})=
(z_2^{(k-1)},\ldots,z_{n}^{(k-1)},z_n^{(k)}).
\end{equation}
Let ${j_0} := {j_0}(\mathbf{z}^{(k-1)})$.
We distinguish the following cases.

\medskip
\begin{enumerate}
\renewcommand{\labelenumi}{(\roman{enumi})}
\itemsep3pt
\item
Let $j_0>i_0+1$.
By \eqref{s3}, it follows that
\begin{eqnarray*}
\min_{i_0+1\leq j\leq n}\nu_p(z_j^{(k)}) &\leq&
\min\{\nu_p(z_{j_0}^{(k-1)}),\nu_p(z_n^{(k)})\}\\
&=&\min\{\nu_p(z_{{j_0}-1}^{(k)}),\nu_p(z_n^{(k)})\}
\qquad(\mbox{by the definition of }j_0)\\
&<&0.
\end{eqnarray*}
Thus ${j_0}(\mathbf{z}^{(k)})={j_0}-1$ or ${j_0}(\mathbf{z}^{(k)})=n$. Both of these
inequalities imply that
\begin{equation}
\label{ss2}
i_0+1\leq{j_0}(\mathbf{z}^{(k)})\leq n.
\end{equation}
\item
Let
$j_0=i_0+1$.
From \eqref{UV} it follows that
$$
\begin{aligned}
U_k&:=a_n z_1^{(k-1)}+\cdots+a_1 z_{n}^{(k-1)}\quad\mbox{and}\\
V_k&:=v_1z_{2}^{(k-1)}+\cdots+v_{n-1}z_{n}^{(k-1)}.
\end{aligned}
$$
The definitions of
$i_0$ and ${j_0}$ imply that
\begin{align*}
\nu_p(v_{i_0})        &\leq0,&\nu_p(v_i)          &>   \makebox[2.0cm][l]{$\nu_p(v_{i_0})$}        \qmboxq{for}i<i_0,\\
                      &      &\nu_p(v_i)          &\geq\makebox[2.0cm][l]{$\nu_p(v_{i_0})$}        \qmboxq{for}i\geq i_0,\\
\intertext{and}                      
\nu_p(z_{j_0}^{(k-1)})&<   0,&\nu_p(z_j^{(k-1)})  &>   \makebox[2.0cm][l]{$\nu_p(z_{j_0}^{(k-1)})$}\qmboxq{for}j>j_0,\\
                      &      &\nu_p(z_j^{(k-1)})  &\geq\makebox[2.0cm][l]{$\nu_p(z_{j_0}^{(k-1)})$}\qmboxq{for}j\leq j_0.
\end{align*}
Hence,
$$
\label{i0j0}
%\nu_p(V_k)=
\nu_p(v_{i_0}z_{i_0+1}^{(k-1)})
=\nu_p(v_{i_0}z_{j_0}^{(k-1)})
<\nu_p(v_iz_{i+1}^{(k-1)})
\qmboxq{for}i\not=i_0.
$$
By \eqref{non-archimedean}, it follows that
$$
\nu_p(V_k)=
\nu_p(v_1z_{2}^{(k-1)}+\cdots+v_{n-1}z_{n}^{(k-1)})=
\nu_p(v_{i_0}z_{i_0+1}^{(k-1)})\leq\nu_p(z_{j_0}^{(k-1)}),
$$
and therefore, by Proposition~\ref{lemabs},
$$
\nu_p(z_n^{(k)})=\nu_p(U_k-\{U_k+V_k\}_p)=\nu_p(\{V_k\}_p)\leq\nu_p(z_{j_0}^{(k-1)}).
$$
Thus,
\begin{equation}
\label{ss3}
i_0+1\leq{j_0}(\mathbf{z}^{(k+1)})\leq n
\end{equation}
and we are done also in this case.
\end{enumerate}
Therefore, the necessity of \eqref{aibound} is proved.

\smallskip
Assume now that \eqref{propf} holds but \eqref{aindn0} does not hold, i.e.
$-\mathbf{a}\not\in\mathcal{D}_n^0.$
By \eqref{dn0}, there exists a $\mathbf{z}\in\Z^n$ such that $\tilde\tau_\mathbf{-a}^k(\mathbf{z})$
does not end up in $\mathbf{0}$. Therefore,
$$
\tau_\mathbf{-a}^k(-\mathbf{z})=-\tilde\tau_\mathbf{-a}^k(\mathbf{z})
$$
does not end up in $\mathbf{0}$. Thus, $\varphi^{-1}(-\mathbf{z})\in \A_p[\beta]$ does not have a finite expansion.
This is a contradiction to \eqref{propf} and
thus, \eqref{aindn0} is a necessary condition.
\end{proof}
\begin{rem}
Condition \eqref{aindn0} indicates a relation to canonical number systems.
In order to explain this connection in more detail, we recall Example~\ref{xpl}.
If the polynomial
$$
\tilde g(x)=\tilde a_n x^n+\tilde a_{n-1}x^{n-1}+\cdots+\tilde a_1 x+p^k\in\Z[x],
$$
gives raise to a canonical number system, then by \cite[Theorem~5.3]{SSTvW},
$$
\left(\frac{\tilde a_n}{p^k},\ldots,\frac{\tilde a_1}{p^k}\right)\in\mathcal{D}_n^0.
$$
Thus, the archimedean roots of
$$
\tilde f(x):=x^n\tilde g(\tfrac1x)=p^k x^n+\tilde a_1 x^{n-1}+\cdots+\tilde a_{n-1}x+\tilde a_n
$$
are strictly inside the complex unit circle. If we define
$$
a_j:=-\frac{\tilde a_j}{p^k}
$$
for $j\in\{1,\ldots,n\}$,
then the polynomial
$$
x^n-a_1 x^{n-1}-\cdots-a_n
$$
fulfills \eqref{aindn0}.
\end{rem}
% ----------------------------------------------------------------
%\nocite{Boyd:87}
%\nocite{Schmidt:77}
\bibliographystyle{plain}
%\bibliography{ffpisot,p-adic-beta}

\end{document}